\documentclass[11pt]{article}

\usepackage[english]{babel}
\usepackage[utf8x]{inputenc}
\usepackage[colorinlistoftodos]{todonotes}
\usepackage{float}

\usepackage{amsmath,amssymb,amsthm,cleveref}
\usepackage{graphicx}
\usepackage{mathrsfs}

\renewcommand{\baselinestretch}{1.25}

\newtheorem{theorem}{Theorem}[section]
\newtheorem{lemma}[theorem]{Lemma}
\newtheorem{proposition}[theorem]{Proposition}
\newtheorem{corollary}[theorem]{Corollary}
\newtheorem{claim}[theorem]{Claim}
\newtheorem{observation}[theorem]{Observation}
\newtheorem{conjecture}[theorem]{Conjecture}

\theoremstyle{definition}
\newtheorem{definition}[theorem]{Definition}
\newtheorem{example}[theorem]{Example}
\theoremstyle{remark}
\newtheorem{remark}[theorem]{Remark}

\numberwithin{equation}{section}

 \newcommand{\To}{\longrightarrow}
 \newcommand{\Map}[3]{#1\, :\, #2\To #3}

 \newcommand{\Nat}{\mathbb{N}}

 \newcommand{\set}[1]{\left\{#1\right\}}
 \newcommand{\Set}[2]{\set{#1\, :\, #2}}

\usepackage{url}

\title{The Path Partition Conjecture and 3-(quasi-)Transitive Digraphs}

\author{Shamik Ghosh\thanks{Department of Mathematics, Jadavpur University, 
        Kolkata, India. Corresponding author\\ (email: ghoshshamik@yahoo.com) 
    }\and Jing Huang
\thanks{Department of Mathematics and Statistics,
University of Victoria, Victoria, B.C., Canada.\\ (email: huangj@uvic.ca). Research
supported by NSERC}}
\date{}

\begin{document}

\maketitle
% ----------------------------------------------------------------
\begin{abstract}
A digraph $D = (V,E)$ is 3-transitive (respectively, 3-quasi-transitive) if whenever
$(u,v,w,z)$ is a directed path of length 3, $(u,z) \in E$ (respectively, $(u,z) \in E$
or $(z,u) \in E$). These digraphs extend the concepts of 
transitive and quasi-transitive digraphs respectively. For 3-transitive and 3-quasi-transitive digraphs, only the strong case has been characterized. 

In this paper we introduce the notion of semi-acyclic digraphs and use it to describe
completely the structure of 3-transitive and 3-quasi-transitive digraphs. 
 As an application of this type of structure description, we prove that the path partition conjecture holds 
for the class of extended semi-acyclic digraphs, which generalizes two previously known 
results on the conjecture.\\
\noindent \textbf{Keywords:} 3-transitive digraph, 3-quasi-transitive digraph, 
semi-acyclic digraph, structure characterization, path partition conjecture
\end{abstract}

\section{Introduction}

A digraph $D = (V,E)$ is {\em transitive} (respectively, {\em quasi-transitive}) 
if whenever $(x, y, z)$ is a directed path of length 2, $(x,z) \in E$ (respectively, 
$(x,z) \in E$ or $(z,x) \in E$). Extending the notions of transitive and quasi-transitive digraphs, Hern\'andez-Cruz 
\cite{HC} and Bang-Jensen \cite{JBJ} introduced 3-transitive and 3-quasi-transitive 
digraphs. A digraph $D = (V,E)$ is called {\em 3-transitive} (respectively, 
{\em 3-quasi-transitive}) if whenever $(u,v,w,z)$ is a directed path of length 3, 
$(u,z) \in E$ (respectively, $(u,z) \in E$ or $(z,u) \in E$). It is worth noting that
every transitive digraph is 3-transitive but not all quasi-transitive digraphs are 
3-quasi-transitive. 3-transitive digraphs are also studied by the name of bitransitive graphs in the context of bipartite digraphs \cite{DGGS}.

The purposes of this paper are two fold. First, we provide structural descriptions of 
3-transitive and 3-quasi-transitive digraphs. These descriptions characterize these
two classes of digraphs in terms of a common structure called semi-acyclic digraphs.
Previously only strong 3-transitive and strong 3-quasi-transitive digraphs were 
characterized. Second, we prove that the path partition conjecture which is widely open
for general digraphs holds for extended semi-acyclic digraphs. Extended semi-acyclic 
digraphs strictly contain the classes of extended acyclic digraphs and 
3-quasi-transitive digraphs (and hence 3-transitive digraphs). Thus, our result 
generalizes the previously known results which assert that the path partition conjecture
holds for extended acyclic digraphs and 3-quasi-transitive digraphs.  

Transitive and quasi-transitive digraphs can be obtained from semicomplete and acyclic
transitive digraphs by substitutions.
Let $H$ be a digraph on vertices $v_1, v_2, \dots, v_r$ and let $S_1, S_2, \dots, S_r$ 
be (vertex-disjoint) digraphs. We use $H[S_1, S_2, \dots, S_r]$ to denote the digraph 
$D$ with vertex set 
\[V(S_1) \cup V(S_2) \cup \cdots \cup V(S_r)\]
and arc set
\[\bigcup\limits_{i=1}^r E(S_i) \bigcup \{(u,v):\ u \in V(S_i), v \in V(S_j), 
    (v_i,v_j) \in E(H)\}.\]
That is, $D$ is obtained from $H$ by {\em substituting} $S_1, S_2, \dots, S_r$ for 
the vertices $v_1, v_2, \dots, v_r$.

\begin{theorem} \label{2-trans}
Let $D$ be a transitive digraph. Then $D = H[S_1, S_2, \dots, S_r]$ where $H$ is 
an acyclic transitive digraph and $S_i$ is a complete digraph for each 
$i = 1, 2, \dots, r$.
\end{theorem}

\begin{theorem} \cite{BH} \label{2-quasi-trans}
Let $D$ be a quasi-transitive digraph. 
\begin{itemize}
\item If $D$ is non-strong, then $D = H[S_1, S_2, \dots, S_r]$ where $r \geq 2$, $H$ is 
an acyclic transitive digraph, and $S_i$ is a strong quasi-transitive digraph for each
$i = 1, 2, \dots, r$.
\item If $D$ is strong, then $D = H[S_1, S_2, \dots, S_r]$ where $H$ is a strong
semicomplete digraph and $S_i$ is either a vertex or a non-strong quasi-transitive
digraph for each $i = 1, 2, \dots, r$.
\end{itemize}
\end{theorem} 

The descriptions of transitive and quasi-transitive digraphs as shown in the theorems 
above are particularly useful in the design of efficient algorithms for solving 
combinatorial problems (e.g., the Hamiltonian cycle problem \cite{BH}). Moreover, these digraphs are interesting because of their relation to comparability graphs and partially ordered sets. The underlying graphs of transitive 
digraphs and quasi-transitive digraphs are {\em comparability graphs} \cite{gallai,ghouila,HH-lex,H}. 

In this paper, we provide structure descriptions analogous to Theorems~\ref{2-trans} and 
\ref{2-quasi-trans} for 3-transitive and 3-quasi-transitive digraphs. Our descriptions 
use a common structure, namely, semi-acyclic digraphs (defined below), and have potential
for solving combinatorial problems for these two classes of digraphs. 

\begin{definition}
A digraph is called {\em semi-acyclic} if it has no directed cycle of length other than 
$2$ and the $2$-cycles are pairwise vertex-disjoint. For convenience we use 
$\langle x,y\rangle$ to denote two cycle with endpoints $x, y$ (i.e., $(x,y), (y,x)$ 
are {\em symmetric} arcs). 
\end{definition}

Let $H$ be a semi-acyclic digraph on vertex set 
$\{x_i, y_i: i \in I\} \cup \{z_j: j \in J\}$
where $\langle x_i, y_i\rangle$, $i \in I$ are the 2-cycles in $H$.
For each $i \in I$, let $T_i$ be a bipartite digraphs with bipartition $(X_i,Y_i)$.
For each $j \in J$, let $Z_j$ be a digraph.
We say that $X_i, Y_i, Z_j$ correspond to $x_i, y_i, z_j$ respectively.
Denote $\mathscr{X} = \{X_i: i \in I\}$, $\mathscr{Y} = \{Y_i: i \in I\}$ and
$\mathscr{Z} = \{Z_j: j \in J\}$. 
We use $H[\mathscr{X},\mathscr{Y},\mathscr{Z}]$ to denote the digraph whose vertex set
is
\[\bigcup\limits_{i\in I} (X_i \cup Y_i) \bigcup\limits_{j\in J} Z_j.\]
The arc set of $D$ consists of all arcs in
\[\bigcup\limits_{i \in I} E(T_i) \bigcup\limits_{j \in J} E(Z_j)\]
and all arcs defined below: If $(u,v)$ is a non-symmetric arc in $H$ where $u, v$ 
correspond $M, N \in \mathscr{X} \cup \mathscr{Y} \cup \mathscr{Z}$ respectively, 
then all arcs from $M$ and $N$ are present in $D$. We say that $D$ is obtained from $H$ 
by {\em substituting} bipartite digraphs $T_i = (X_i,Y_i)$ (for the 2-cycles) and 
digraphs $Z_j$ (for the vertices not on 2-cycles).

We will prove the following two theorems:

\begin{theorem} \label{main1}
Let $D$ be a 3-transitive digraph. Then $D$ is obtained from a 3-transitive semi-acyclic
digraph $H$ by substituting complete bipartite digraphs for the 2-cycles and complete 
digraphs or directed triangles with at most two symmetric arcs for the vertices not
on 2-cycles.
\end{theorem}

\begin{theorem} \label{main2}
Let $D$ be a 3-quasi-transitive digraph. Then $D$ is obtained from a 3-transitive
semi-acyclic digraph $H$ by substituting strong semicomplete bipartite digraphs for
the 2-cycles and strong semicomplete digraphs or the digraphs $F_n$ in Figure~\ref{fnn1}
for the vertices not on 2-cycles.
\end{theorem}

Let $D = (V,E)$ be a digraph. For a subset $\emptyset\neq X\subseteq V$, let $D[X]$
denote the subdigraph of $D$ induced by $X$. For a path \(P\) in $D$, let \(|P|\)
denote its order, that is, the number of vertices of \(P\). Let
\[
        \lambda(D)=\max\{|P|: P \text{ is a path in }D\}.
\]
The following conjecture is known as the {\em path partition conjecture}.

\begin{conjecture} \label{ppc}
For every digraph $D = (V,E)$ and every pair of positive integers 
$\lambda_1, \lambda_2$ satisfying $\lambda_1+\lambda_2=\lambda(D)$, there is a partition
$V=A\cup B$ such that 
$$ \lambda(D[A])\leq\lambda_1,\qquad \lambda(D[B])\leq\lambda_2.$$
\end{conjecture}

The path partition conjecture remains wide open although it has been verified for 
some classes of digraphs. Arroyo and Galeana-S\'anchez \cite{AG} proved that every
strong 3-quasi-transitive digraph satisfies the path partition conjecture. 
They also showed that the conjecture holds for digraphs which can be obtained from 
acyclic digraphs by substituting digraphs that meet the conjecture. 
Bo et. al. \cite{BGLL} proved that every 3-quasi-transitive digraph satisfies the path 
partition conjecture.

We will prove the following theorem which generalizes the above mentioned results on
the path partition conjecture.

\begin{theorem} \label{main3}
Let $D$ be obtained from a semi-acyclic digraph $H$ by substituting bipartite digraphs 
$T_i$ for the 2-cycles and digraphs $Z_j$ for the vertices not on 2-cycles. If
the digraphs $Z_j$ satisfy the path partition conjecture, then $D$ satisfies 
the path partition conjecture.
\end{theorem}

Let $D=(V,E)$ be a digraph. Call $D$ {\em strong} if for any two vertices $u, v$,
there is a directed path from $u$ to $v$ and a directed path from $v$ to $u$.
For $X,Y\subseteq V$, if $xy\in E$ for all $x\in X$ and $y\in Y$, we denote $X\longmapsto Y$. For $X=\set{x}$, we write $x\longmapsto Y$. The notation $X\longmapsto y$ is used for $Y=\set{y}$. Let $x\in V$. Then $N^+(x)=\Set{y\in V}{xy\in E}$, $N^-(x)=\Set{y\in V}{yx\in E}$, $N^{++}(x)=\Set{y\in V}{xw,wy\in E\text{ for some }w\in V}$ and $N^{--}(x)=\Set{y\in V}{yw,wx\in E\text{ for some }w\in V}$.

\section{3-Transitive digraphs}

\begin{lemma} \cite{HC} \label{strong-3-trans}
Let $D$ be a 3-transitive digraph. If $D$ is strong then it is either complete or 
complete bipartite or a directed triangle with at most two symmetric arcs.
\end{lemma}

Let $D=(V,E)$ be a $3$-transitive digraph and let $\mathscr{S}$ be the set of strong 
components of $D$. By Lemma \ref{strong-3-trans}, $\mathscr{S}$ can be
partitioned into $\mathscr{S}_1$, $\mathscr{S}_2$ and $\mathscr{S}_3$ where
\begin{tabbing} 
\hspace{1cm} \= $\mathscr{S}_1 =\Set{S\in\mathscr{S}}{|S|=1}$\\
\> $\mathscr{S}_2=\Set{S\in\mathscr{S}}{S=X_S\cup Y_S\text{ is complete bipartite, }|S|>1}$ \\ 
\> $\mathscr{S}_3=\Set{S\in\mathscr{S}}{S\text{ is either complete or a directed triangle
with at most two symmetric arcs, }|S|>2}.$
\end{tabbing}

We define a digraph $D^\# =(V^\# ,E^\# )$ with the vertex set
$$V^\# =\mathscr{S}_1\cup\Set{X_S}{S\in\mathscr{S}_2}\cup \Set{Y_S}{S\in\mathscr{S}_2}\cup\mathscr{S}_3,$$
and the arc set
$$E^\# = \{(M,T): M, T \in V^\#, M \neq T,\ \mbox{and}\ x\to y\ \mbox{for\ some}\ 
        x \in M, y \in T\}.$$
We call $D^\#$ the {\em skeleton} of $D$. 

\begin{lemma}\label{lemhash}
Let $D=(V,E)$ be a $3$-transitive digraph and let $D^\#$ be the skeleton of $D$.
Suppose that $M\in V^\#$ and $x \in V\setminus M$. Then
\begin{enumerate}
\item $x\rightarrow z$ in $D$ for some $z\in M$ $\Longrightarrow$ $x\longmapsto M$.
\item $z\rightarrow x$ in $D$ for some $z\in M$ $\Longrightarrow$ $M \longmapsto x$.
\end{enumerate}
\end{lemma}
\begin{proof}
Let $x\rightarrow z$ for some $z\in M$. The statement is trivial for $M\in\mathscr{S}_1$. Let $M\in \mathscr{S}_3$. Then $M$ is either complete with at least $3$ vertices or a directed triangle with at most two symmetric arcs. In the first case, for any $y\in M$, we have $x\rightarrow z\rightarrow y^\prime\rightarrow y$ for some $y^\prime\in M$ such that $y^\prime\neq\set{y,z}$. Since $|M|\geq 3$, such a $y^\prime$ always exists. Thus $x\rightarrow y$ in $D$. In the second case, let $(z,z_1,z_2,z)$ be the directed triangle. Then $x\rightarrow z\rightarrow z_1 \rightarrow z_2$ holds in $D$ which implies $x\rightarrow z_2$ in $D$. Also $x\rightarrow z_2\rightarrow z\rightarrow z_1$ in $D$ implies $x\rightarrow z_1$ in $D$. Thus $x\rightarrow y$ for all $y\in M$.

Next, let $M=X_S$ For some $S\in\mathscr{S}_2$. If $|X_S|=1$, we have nothing to prove. Let $|X_S|>1$ and $y\in X_S$. Since $Y_S\neq\emptyset$, let $z_1\in Y_S$. Then $x\rightarrow z\rightarrow z_1\rightarrow y$ in $D$ as $S$ is complete bipartite. This implies $x\rightarrow y$. The case is similar for $M=Y_S$.

Suppose $z\rightarrow x$ for some $z\in M$. As before, the result is trivial for $M\in\mathscr{S}_1$. Let $M\in\mathscr{S}_3$. If $M$ is complete with at least $3$ vertices, then for any $y,y_1\in M\smallsetminus\set{z}$, we have $y\rightarrow y_1\rightarrow z\rightarrow x$ in $D$ which implies $y\rightarrow x$ in $D$. Again let $(z,z_1,z_2,z)$ be the directed triangle in $M$. Then $z_1\rightarrow z_2\rightarrow z\rightarrow x$ in $D$ and so $z_1\rightarrow x$ in $D$. Thus $z_2\rightarrow z\rightarrow z_1\rightarrow x$ in $D$ and so $z_2\rightarrow x$ in $D$.

Finally, let $M=X_S$ and $|M|>1$. As $Y_S\neq\emptyset$, take $y^\prime\in Y_S$. Let $y\in M=X_S$. Then $y\rightarrow y^\prime\rightarrow z\rightarrow x$ in $D$ which implies $y\rightarrow x$ in $D$ as required. The proof for $M=Y_S$ is similar.
\end{proof}

\begin{lemma} \label{sa1}
Let $D$ be a $3$-transitive digraph. Then $D^\#$ is a semi-acyclic $3$-transitive 
digraph.
\end{lemma}
\begin{proof}
Let $M,T\in V^\#$ such that $M\neq T$. If $M=X_S$ and $T=Y_S$ (or vice-versa) for some $S\in\mathscr{S}_2$, then $M\leftrightarrow T$. In all other cases, $M$ and $T$ correspond to different strong components. Thus $x\rightarrow y$ and $y\rightarrow x$ cannot happen simultaneously in $D$ for otherwise union of these components would be a strong subdigraph of $D$ which is a contradiction. The same thing happens if we assume a directed cycle in $D^\#$ of length more than $2$. Also since strong components are vertex-disjoint, we have $D^\#$ is semi-acyclic.

Let $X\rightarrow Y\rightarrow Z\rightarrow W$ in $D^\#$. Then there are $x\in X$, $y,y^\prime\in Y$, $z,z^\prime\in Z$ and $w\in W$ such that $x\rightarrow y$, $y^\prime\rightarrow z$, $z^\prime \rightarrow w$ in $D$. By Lemma \ref{lemhash}, we have $x\rightarrow y\rightarrow z\rightarrow w$ in $D$. Thus $x\rightarrow w$ in $D$ which implies $X\rightarrow W$ in $D^\#$. Therefore $D^\#$ is $3$-transitive.
\end{proof}

We now prove Theorem~\ref{main1} which is reiterated as follows.

\begin{theorem} \label{3-trans}
Let $D$ be a $3$-transitive digraph. Then there exists a semi-acyclic 3-transitive 
digraph $H$ such that $D = H[\mathscr{X},\mathscr{Y},\mathscr{Z}]$ where
\begin{itemize}
\item $\mathscr{X} = \{X_i: i \in I\}$, $\mathscr{Y} = \{Y_i: i \in I\}$ and
      $\mathscr{Z} = \{Z_j: j \in J\}$.
\item $(X_i, Y_i)$ is a complete bipartite digraph for each $i \in I$.
\item $Z_j$ is either a complete digraph or a directed triangle with at most two 
      symmetric arcs. 
\end{itemize}
\end{theorem}
\begin{proof}
Let $H = D^\#$, i.e., the skeleton of $D$, and let 
\begin{itemize}
\item $\{(X_i,Y_i): i\in I\} = \mathscr{S}_2$.
\item $\{Z_j: j\in J\} = \mathscr{S}_1 \cup \mathscr{S}_3$.
\end{itemize} 
Then $H$ is a semi-acyclic 3-transitive digraph by Lemma \ref{sa1}, and 
$D = H[\mathscr{X},\mathscr{Y},\mathscr{Z}]$ by Lemma \ref{lemhash}. 
\end{proof}

In the following, we show how to obtain a semi-acyclic $3$-transitive digraph from 
an acyclic $3$-transitive digraph. We begin with defining some special arcs.

\begin{definition}[Flexible arc]
Let $D=(V,E)$ be an acyclic $3$-transitive digraph. An arc $uv\in E$ is called 
{\em flexible} if replacing the arc $uv$ by $vu$ the digraph remains to be acyclic and 
$3$-transitive.
\end{definition}

\begin{proposition}\label{flprop}
Let $D=(V,E)$ be an acyclic $3$-transitive digraph. Then $uv\in E$ is flexible if and only if the following conditions hold:
\begin{enumerate}
\item $(u,v)$ is the only directed path from $u$ to $v$.
\item $N^-(v)\longmapsto N^+(u)$.
\item $N^{--}(v)\longmapsto u$.
\item $v\longmapsto N^{++}(u)$.
\end{enumerate}
\end{proposition}

\begin{proof}
Let $D=(V,E)$ be an acyclic $3$-transitive digraph and $uv\in E$. Let $D^\prime=(V,E^\prime)$, where $E^\prime =(E\smallsetminus\set{uv})\cup\set{vu}$. Suppose $uv\in E$ is a flexible arc in $D$. Then $D^\prime$ is acyclic and $3$-transitive.

If there is a directed path $(u=x_0,x_1,x_2,\ldots,x_n=v)$ ($n>1$) from $u$ to $v$ other than $(u,v)$, then after replacing $uv$ by $vu$, we have a directed cycle $(u=x_0,x_1,x_2,\ldots,x_n=v,u)$ and so $D^\prime$ is not acyclic. Therefore $uv$ is the only directed path from $u$ to $v$.

Let $x\in N^-(v)$, $y\in N^+(u)$. Then $xv,vu,uy\in E^\prime$, then $xy\in E^\prime$ as $D^\prime$ is $3$-transitive. Then $xy\in E$ as $\set{x,y}\neq \set{u,v}$. So we have $N^-(v)\longmapsto N^+(u)$ in $D$. Next, if $x\in N^{--}(v)$, then $xw,wv,vu\in E^\prime$ for some $w\in V$. This implies $xu\in E^\prime$ and so $xu\in E$ as $x\neq v$. Thus $N^{--}(v)\longmapsto u$. Finally, if $x\in N^{++}(u)$, then $vu,uw,wx\in E^\prime$ for some $w\in V$, which implies $vx\in E^\prime$ and hence $vx\in E$ as $x\neq u$. Thus we have $v\longmapsto N^{++}(u)$.

Conversely, let $D=(V,E)$ be an acyclic $3$-transitive digraph and $uv\in E$ satisfying the given conditions. Let $D^\prime=(V,E^\prime)$, where $E^\prime =(E\smallsetminus\set{uv})\cup\set{vu}$. Now, since $D$ had no directed cycle, $D^\prime$ would also be free from directed cycle unless the arc $vu$ creates one. But then there is a directed path from $u$ to $v$ in $D$ other than $(u,v)$ which violets the first condition. Therefore $D^\prime$ is acyclic. 

All the three other conditions ensures that $D$ remains to be $3$-transitive after replacing the arc $uv$ by $vu$ to obtain $D^\prime$. In fact, let $x,y,z,w\in V$ such that $xy,yz,zw\in E^\prime$. If $vu\not\in\set{xy,yz,zw}$, then $xy,yz,zw\in E$. So $xw\in E$. Now $xw\neq uv$ for otherwise we would have a directed path $(u,y,z,v)$ which again violets the first condition. Thus $xw\in E^\prime$. Now, if $vu=xy$, then $w\in N^{++}(u)$ and since $v\longmapsto N^{++}(u)$, we have $xw=vw\in E$ and hence in $E^\prime$ also as $vw\neq uv$. Next, if $vu=yz$, then $x\in N^-(v)$ and $w\in N^+(u)$ and so $xw\in E$ as $N^-(v)\longmapsto N^+(u)$. Again $xw=uv$ contradicts the first condition. So we have $xw\in E^\prime$. Finally, if $vu=zw$, then $x\in N^{--}(v)$ and we have $xw=xu\in E$ as $N^{--}(v)\longmapsto u$. Since $xw=xu\neq uv$, we have $xw\in E^\prime$. Therefore $D^\prime$ is $3$-transitive as required.
\end{proof}

It is important to understand that, for any arc $uv$ in an acyclic digraph, there is no directed path from $v$ to $u$ in $D$ for otherwise there would be a directed cycle. Thus, if $uv$ is a flexible arc in an acyclic $3$-transitive digraph $D$, then  neither there is a directed path from $u$ to $v$ other than $uv$ nor there is a directed path from $v$ to $u$ in $D$. Therefore end points of a flexible arc in $D$ cannot be the end points of a $3$-path $x\rightarrow y\rightarrow z\rightarrow w$ in $D$.

\begin{example}
It is not necessary that every acyclic $3$-transitive digraph contains a flexible arc. The following example shows that none of the arcs of the digraph $D=(V,E)$ in Figure \ref{nflex} is flexible, though $D$ is acyclic and $3$-transitive. Here $V=\set{a,b,c,d,e}$ and $E=\set{ba,bc,ca,ce,dc,de}$. Note that reversing $ba$ or $de$ invites directed cycles. If we reverse $bc$ or $dc$, then we will loose $3$-transitivity (as $da$ or $be$ is missing respectively). Similar is the case for $ca$ or $ce$. 
\end{example}

\begin{figure}[ht]
\begin{center}
\includegraphics[scale=0.7]{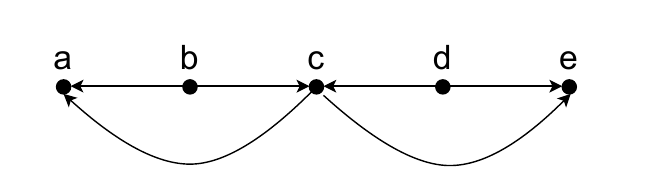}
\caption{An acyclic $3$-transitive digraph $D$ with no flexible arcs}\label{nflex}
\end{center}
\end{figure}

Unfortunately, any collection of vertex disjoint flexible arcs in an acyclic $3$-transitive digraph cannot be transformed to $2$-cycles to obtain a semi-acyclic $3$-transitive graphs. We require one more condition to complete the construction.

\begin{definition}[Faithful matching]
Let $D=(V,E)$ be an acyclic $3$-transitive digraph. A matching $M$ of $D$ consisting of flexible arcs is called {\em faithful} if $M$ satisfies the following property:
\begin{equation}\label{extram}
yx,wz\in M,\ yz\in E\Longrightarrow xw\in E.
\end{equation}
\end{definition}

\begin{example}
Consider the digraph $D=(V,E)$, where $D=\set{a,b,c,d}$ and $E=\set{ba,bc,dc}$, i.e., $D$ is the digraph
$$a\longleftarrow b\longrightarrow c\longleftarrow d$$
Then $D$ is acyclic and $3$-transitive. Also $M=\set{ba,dc}$ is a matching in $D$ consisting of flexible arcs. But $M$ is not faithful as $ad\not\in E$. Note that the digraph
$$a\longleftrightarrow b\longrightarrow c\longleftrightarrow d$$
is not $3$-transitive. 
\end{example}

\begin{lemma}
Let $D=(V,E)$ be a semi-acyclic $3$-transitive digraph. Let $\mathscr{S}$ be the collection of $2$-cycles in $D$. For each $C \in\mathscr{S}$, we replace $C = (u,v,u)$ by an arc $uv$, where $u$ and $v$ are chosen arbitrarily. Then the digraph $D^\prime=(V,E^\prime)$ obtained by this is an acyclic $3$-transitive digraph where the set $M$ of replaced arcs form a faithful matching in $D^\prime$.
\end{lemma}

\begin{proof}
We first note that $D^\prime$ is acyclic by construction as we removed all $2$-cycles and since $D$ has no other cycle, removing arcs cannot create further cycles in $D^\prime$. Suppose $x,y,z,w\in V$ such that $xy,yz,zw\in E^\prime$. Since $E^\prime\subseteq E$ and $D$ is $3$-transitive, $xw\in E$. If $(x,w,x)$ is a $2$-cycle in $D$, then $(x,y,z,w,x)$ is a cycle in $D$ which is a contradiction as $D$ is semi-acyclic and so cannot have a cycle of length more than $2$. Thus $(x,w,x)$ is not a $2$-cycle in $D$ and hence $xw$ is not deleted to form $D^\prime$, i.e., $xw\in E^\prime$. Therefore $D^\prime$ is $3$-transitive. Now since $2$-cycles in $D$ are vertex disjoint, $M$ is a matching in $D^\prime$. 

Let $(u,v,u)$ be a $2$-cycle in $D$ and we replaced $(u,v,u)$ by the arc $uv$, i.e., we removed the arc $vu$ in $D^\prime$. We show that $uv$ is a flexible arc in $D^\prime$. Since $E^\prime\subseteq E$, if there is another directed path from $u$ to $v$ in $D^\prime$, then there would be the same in $D$ and this path along with $vu$ would form a directed cycle of length more than $2$. But this is not possible as $D$ is semi-acyclic. Thus the first condition of Proposition \ref{flprop} is satisfied. The other $3$ conditions follow
from the definition of the 3-transitivity and the fact $uv,vu\in E$.

Now let $x,y,z,w\in V$ such that $yx,wz\in M$ and $yz\in E^\prime$. Then we have $xy,zw\in E$. Also as $E^\prime\subseteq E$, $yz\in E$. Since $D$ is $3$-transitive, we have $xw\in E$. Further, as $yx,wz\in M$, $xw$ or $wx$ cannot belong to $M$. Thus $xw\in E^\prime$. This implies $M$ is faithful.
\end{proof}

\begin{theorem}
Let $D=(V,E)$ be an acyclic $3$-trasitive digraph. Suppose $D$ has a faithful matching $M$. Let $E^\prime=E\cup\Set{vu}{uv\in M}$. Then $D^\prime =(V,E^\prime)$ is a semi-acyclic $3$-transitive digraph.
\end{theorem}

\begin{proof}
Let $x,y,z,w\in V$ such that $xy,yz,zw\in E^\prime$. If $xy,yz,zw\in E$, then $xw\in E\subseteq E^\prime$. Suppose $yz\in E^\prime\smallsetminus E$. Then $zy\in M$. Since $M$ is a matching, $xy,zw\not\in M$ and thus $xy,zw\in E$. Now since $zy$ is flexible, by Proposition \ref{flprop}, $N^-(y)\longmapsto N^+(z)$. As $x\in N^-(y)$ and $w\in N^+(z)$, we have $xw\in E\subseteq E^\prime$. Next, let $xy\in E^\prime\smallsetminus E$, i.e., $yx\in M$ and $zw\in E$. Then $yz\in E$ as $xy$ and $yz$ have a common endpoint. So $w\in N^{++}(y)$ which implies $xw\in E$ as $x\longmapsto N^{++}(y)$ holds for the flexible arc $yx$. Similarly, if $xy,yz\in E$ and $wz\in M$, then $xw\in E$ as $N^{--}(z)\longmapsto w$ for the flexible arc $xw$. Finally, if $xy,zw\in E^\prime\smallsetminus E$, then $yx,wz\in M$ and $yz\in E$ and so by (\ref{extram}), we have $xw\in E$. Therefore $D^\prime$ is $3$-transitive.

Suppose $D^\prime$ has a directed cycle of length more than $2$. Since $D^\prime$ is $3$-transitive, we have a directed cycle $C$ in $D^\prime$ of length $3$ or $4$. Suppose $C$ is of length $4$. Now since $D$ is acyclic, $C$ must contain some arcs $vu$ such that $uv\in M$. If $C$ contains only one such arc, then deleting the arc $vu$ from $C$ we obtain a directed path from $u$ to $v$ in $D$ which is a contradiction as $uv$ is a flexible arc in $D$. Again if $C$ contains arcs $vu, yx$ such that $uv,xy\in M$. Then $uy,xv\in E$ as $C=(v,u,y,x,v)$. On the other hand, we have $vx\in E$ by (\ref{extram}), which is a contradiction as $xv\in E$ and $D$ has no $2$-cycle. Finally, let $C=(x,v,u,x)$ is a $3$-cycle and contain an arc $vu$ such that $uv\in M$. Then $ux,xv\in E$ which yields a directed path $(u,x,v)$ in $D$, contradicting that $uv$ is a flexible arc. Therefore $D^\prime$ has only $2$-cycles and since $M$ is a matching, these $2$-cycles are vertex disjoint. Hence $D^\prime$ is semi-acyclic.
\end{proof}

\section{3-Quasi-transitive digraphs}

\begin{lemma} \cite{GGU} \label{strong-3-quasi-trans}
Let $D$ be a 3-quasi-transitive digraph. If $D$ is strong then it is either semicomplete or semicomplete bipartite or a digraph in the class $F_n=(V_n,E_n)$, where $V_n=\set{a,b,d}\cup\Set{c_i}{i\in [n]}$ such that each quadruple $(a,b,c_i,d,a)$ form a $4$-cycle with a diagonal arc $bd$ (see Figure \ref{fnn1}).
\end{lemma}

\begin{figure}[ht]
\begin{center}
\includegraphics[scale=0.7]{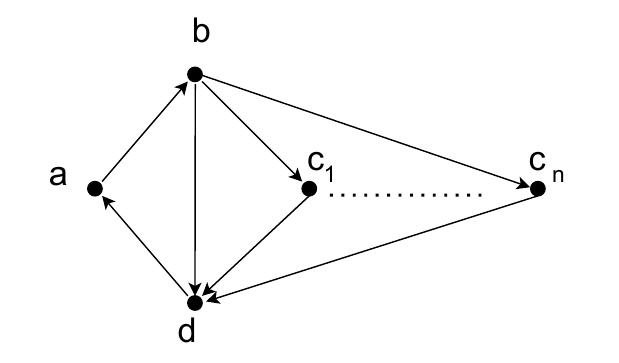}
\caption{Directed graphs $F_n$ for $n\in\Nat$.}\label{fnn1}
\end{center}
\end{figure}

In the following, we provide a description of $3$-quasi-transitive digraphs, in general.
The description is similar to that in the case of $3$-transitive graphs. One important 
fact we use here is that every strong semicomplete digraph is Hamiltonian \cite{CM}.

Let $D=(V,E)$ be a $3$-quasi-transitive digraph and let $\mathscr{S}$ be the set of all 
strong components of $D$. By Lemma \ref{strong-3-quasi-trans}, $\mathscr{S}$ can be 
partitioned into $\mathscr{S}_1$, $\mathscr{S}_2$ and $\mathscr{S}_3$ where
\begin{tabbing} 
\hspace{1cm} \= $\mathscr{S}_1 =\Set{S\in\mathscr{S}}{|S|=1}$\\ 
\> $\mathscr{S}_2=\Set{S\in\mathscr{S}}{S=X_S\cup Y_S\text{ is semicomplete bipartite, }|S|>1}$ \\ 
\> $\mathscr{S}_3=\Set{S\in\mathscr{S}}{S\text{ is either semicomplete or isomorphic to 
    $F_n$, }|S|>2}.$
\end{tabbing} 

We define a digraph $D^\# =(V^\# ,E^\# )$ with the vertex set
$$V^\# =\mathscr{S}_1\cup\Set{X_S}{S\in\mathscr{S}_2}\cup \Set{Y_S}{S\in\mathscr{S}_2}\cup\mathscr{S}_3,$$
and the arc set
$$E^\# = \{(M,T): M, T \in V^\#, M \neq T,\ \mbox{and}\ x\to y\ \mbox{for\ some}\
        x \in M, y \in T\}.$$
We call $D^\#$ the {\em skeleton} of $D$.

\begin{observation}\label{notrev}
Let $D=(V,E)$ be a $3$-quasi-transitive digraph. Let $S,T\in\mathscr{S}$ such that $x\rightarrow y$ for some $x\in S$ and $y\in T$, $S\neq T$. Then $y^\prime\not\rightarrow x^\prime$ for any $x^\prime\in S$ and $y^\prime\in T$.
\end{observation}

\begin{proof}
Since $S$ and $T$ are strong components, both way adjacency would make $S\cup T$ a strong subdigraph of $D$, which is a contradiction as $S\neq T$.
\end{proof}

Let $D=(V,E)$ be a $3$-quasi-transitive digraph and $x,y\in V$. For convenience, we write $x\rightsquigarrow y$ if there is a directed path from $x$ to $y$.

\begin{lemma}\label{quasilem}
Let $D=(V,E)$ be a $3$-quasi-transitive digraph. Let $M\in V^\#$ and $x\in V\setminus$. 
Then
\begin{enumerate}
\item $x\rightarrow y$ for some $y\in M$ $\Longrightarrow x\longmapsto M$.
\item $y\rightarrow x$ for some $y\in M$ $\Longrightarrow M\longmapsto x$.
\end{enumerate}
\end{lemma}
\begin{proof}
We first note that, by Observation \ref{notrev}, the reverse direction is impossible throughout the proof. Hence whenever $3$-quasi-transitivity gives adjacency between $x$ and a vertex of $M$, the arc must be directed from $x$ to that vertex in case (1) and the opposite in case (2). 

{\bf (1):}\ If $M\in\mathscr{S}_1$, then we have nothing to prove. Let $M=X_S$ for some $S\in\mathscr{S}_2$. Let $x\in V\smallsetminus S$ such that $x\rightarrow y$ for some $y\in M$. Let $z\in M$. Since $S$ is a strong component of $D$ and $S=X_S\cup Y_S$ is bipartite, we have there is a directed path $y=y_0\rightarrow y_1\rightarrow y_2\rightarrow\cdots\rightarrow y_{2k}=z$ for some $k\in\Nat$, where $y_i\in Y_S$ for odd $i$ and $y_i\in X_S$ for even $i$. Now we have $x\rightarrow y\rightarrow y_1\rightarrow y_2$. Then we have $x\rightarrow y_2$ by the comment made at the beginning. If $k=1$, we are done. Otherwise, we have $x\rightarrow y_2\rightarrow y_3\rightarrow y_4$. By the same argument, we have $x\rightarrow y_4$. We repeat the process until we reach $x\rightarrow y_{2k}=z$ as required. The case is similar for $M=Y_S$ for some $S\in \mathscr{S}_2$. 

Let $M\in \mathscr{S}_3$. Then $M$ is a strong semicomplete digraph or $M$ is isomorphic to $F_n$ for some $n\in\Nat$. First suppose $M$ is strong and semicomplete. Let $x\in V\smallsetminus M$ such that $x\rightarrow y$ for some $y\in M$. Since $M$ is strong and semicomplete, we have $M$ is Hamiltonian. Let $y=y_0\rightarrow y_1\rightarrow y_2\rightarrow\cdots\rightarrow y_k\rightarrow y$ be a Hamiltonian cycle in $M$ with $k\geq 2$ (as $|M|>2$). Now as before, we have $x\rightarrow y_i$ for all even $i$. Suppose $k>2$. Since $M$ is semicomplete, either $y_1\rightarrow y_3$ or $y_3\rightarrow y_1$. If $y_1\rightarrow y_3$, then $x\rightarrow y\rightarrow y_1\rightarrow y_3$ which implies $x\rightarrow y_3$ and hence by rotating the cycle, we have $x\rightarrow y_i$ for all odd $i$ also. Similar case happens if $y_3\rightarrow y_1$ as $x\rightarrow y_2\rightarrow y_3\rightarrow y_1$. So let $k=2$ and $M$ is a $3$-cycle $(y,y_1,y_2,y)$ with possibly some symmetric arcs. As before we have $x\rightarrow y_2$ and so we have $x\rightarrow y_2\rightarrow y\rightarrow y_1$ implying $x\rightarrow y_1$ as well. 

Finally, let $M\cong F_n$ for some $n\in \Nat$. Then every vertex of $M$ belongs to a $4$-cycle, say, $(a,b,c_i,d,a)$. Then as in the previous case, if we have $x\rightarrow a$, then $x\rightarrow a\rightarrow b\rightarrow c_i$ implies $x\rightarrow c_i$ and $x\rightarrow a\rightarrow b\rightarrow d$ implies $x\rightarrow d$. But then $x\rightarrow d\rightarrow a\rightarrow b$ implies $x\rightarrow b$. Since $c_i$ was taken arbitrarily, we have $x\rightarrow c_i$ for all $i\in [n]$. 

Similarly, $x\rightarrow b$ implies $x\rightarrow a$ as $b\rightarrow d\rightarrow a$ in $F_n$ and $x\rightarrow d$ implies $x\rightarrow b$ as $d\rightarrow a\rightarrow b$ in $F_n$. Lastly, if $x\rightarrow c_i$ for some $i\in [n]$, then we have $x\rightarrow a$ as $c_i\rightarrow d\rightarrow a$ in $F_n$. So in all cases we have $x\rightarrow a$ and hence we have $x\longmapsto M$ as shown in the above paragraph. 

{\bf (2):}\ It is interesting to note that, the converse of a 3-quasi-transitive digraph (obtained from the given digraph by reversing all the arc) is again 3-quasi-transitive and has the same strong components. Thus part (2) follows from part (1) applied to the converse digraph.
\end{proof}

\begin{theorem} \label{sa2}
Let $D$ be a $3$-quasi-transitive graph. Then $D^\#$ is a semi-acyclic $3$-transitive 
digraph.
\end{theorem}

\begin{proof}
Let $M,T\in V^\#=V(D^\#)$. Suppose $M=X_S$, $T=Y_S$ (or, $M=Y_S$, $T=X_S$) for some $S\in\mathscr{S}_2$. Now since $S$ is a strong component, there are $x_1,x_2\in M$ (may not be distinct), $y_1,y_2\in T$ (may not be distinct) such that $x_1\rightarrow y_1$ and $y_2\rightarrow x_2$. Thus we have $M\leftrightarrow T$ in $D^\#$. Otherwise $M$ and $T$ are different strong components or being part of different strong components. Thus if there is $x\rightarrow y$ for some $x\in M$ and $y\in T$, then there cannot be any $x^\prime\in M$, $y^\prime \in T$ such that $y^\prime \rightarrow x^\prime$ by Observation \ref{notrev}. Again, if there is a directed cycle in $D^\#$, then union of all of them would be a subdigraph of the same strong component. Thus the only possible cycle is a $2$-cycle. Also since strong components are vertex disjoint, these $2$-cycles (if they exist) are also vertex disjoint. Therefore $D^\#$ is semi-acyclic.

Let $X\rightarrow Y\rightarrow Z\rightarrow W$ in $D^\#$ with $X\neq W$. If the three skeleton arcs can be realized by a directed path $x\rightarrow y\rightarrow z\rightarrow w$ in D, then by $3$-quasi-transitivity of $D$, we have either $x\rightarrow w$ or $w\rightarrow x$. If $X$ and $W$ are parts of different strong components, then Observation \ref{notrev} gives $x\rightarrow w$ and consequently we get $X\rightarrow W$ in $D^\#$. If $X$ and $W$ are the two parts of the same split semicomplete bipartite component, then $X\leftrightarrow W$ holds.

The only obstruction to such a realization occurs when consecutive skeleton arcs pass back and forth between $X_S$ and $Y_S$ of the same semicomplete bipartite strong component $S$. Let us call a skeleton arc $U\rightarrow V$ in $D^\#$ as {\em special} if it is one of the two arcs between $X_S$ and $Y_S$ for some $S\in\mathscr{S}_2$ and we denote it by $U\dashrightarrow V$ and if the arc is not special, then by $U\longmapsto V$ to mean that we have $u\rightarrow v$ uniformly for all $u\in U$, $v\in V$. 

Now, if there are two consecutive special arcs, then we are just moving back and forth between $X_S$ and $Y_S$. In that case the desired conclusion $X\rightarrow W$ is immediate from one of the already existing arcs $X_S\leftrightarrow Y_S$, or from the fact that the repeated skeleton vertices force the conclusion to be one of the given arcs. For example, if we have, $X\dashrightarrow Y\dashrightarrow Z\longmapsto W$, then $X=Z=X_S$, $Y=Y_S$ (or, $X=Z=Y_S$, $Y=X_S$) for some $S\in\mathscr{S}_2$. Thus $X=Z\longmapsto W$ and so we have $X\rightarrow W$ in $D^\#$. The case $X\longmapsto Y\dashrightarrow Z\dashrightarrow W$ is similar. Let $X\dashrightarrow Y\dashrightarrow Z\dashrightarrow W$, then $X=Z=X_S$, $Y=W=Y_S$ or, $X=Z=Y_S$, $Y=W=X_S$ for some $S\in\mathscr{S}_2$. In any case, $\set{X,W}=\set{X_S,Y_S}$. Hence $X\leftrightarrow W$ in $D^\#$. Therefore $D^\#$ is $3$-transitive as required.
\end{proof}

We now prove Theorem~\ref{main2} which is reiterated as follows.

\begin{theorem} \label{3-quasi-trans}
Let $D$ be a $3$-quasi-transitive digraph. Then there exists a semi-acyclic 3-transitive
digraph $H$ such that $D = H[\mathscr{X},\mathscr{Y},\mathscr{Z}]$ where
\begin{itemize}
\item $\mathscr{X} = \{X_i: i \in I\}$, $\mathscr{Y} = \{Y_i: i \in I\}$ and
      $\mathscr{Z} = \{Z_j: j \in J\}$.
\item $(X_i, Y_i)$ is a strong semicomplete bipartite digraph for each $i \in I$.
\item $Z_j$ is either a strong semicomplete digraph or isomorphic to $F_n$ for some 
      $n \geq 4$.
\end{itemize}
\end{theorem}
\begin{proof}
Let $H$ be the skeleton of $D$, and let $\{(X_i,Y_i): i\in I\} = \mathscr{S}_2$ and $\{Z_j: j\in J\} = \mathscr{S}_1 \cup \mathscr{S}_3$. Then $H$ is a semi-acyclic 3-transitive by Lemma \ref{sa2} and 
$D = H[\mathscr{X},\mathscr{Y},\mathscr{Z}]$ by Lemma \ref{quasilem}.
\end{proof}

\section{Path partition conjecture}

In this section we examine Conjecture~\ref{ppc} for digraphs $D$ of the form
$H[\mathscr{X},\mathscr{Y},\mathscr{Z}]$. We will prove that if $H$ is semi-acyclic and
the substituted digraphs in $\mathscr{Z}$ satisfy the conjecture then $D$ 
satisfies the conjecture. First, we show that the conjecture holds for semi-acyclic
digraphs.

\begin{definition}\label{def:h}
Let $D$ be a semi-acyclic digraph. For each $v\in V(D)$, we define the {\em height function} $\Map{h}{V(D)}{\Nat}$ by
$$
h(v)=\max\Set{|P|}{P\text{ is a path of }D\text{ ending at }v},
$$
where $\Nat$ is the set of natural numbers. Thus $1\leq h(v)\leq\lambda(D)$.
\end{definition}

\begin{lemma}\label{lem:ha2}
Let $D$ be a semi-acyclic digraph and let $h$ be defined as in Definition~\ref{def:h}. Then the following statements hold:
\begin{enumerate}
\item If $uv\in E(D)$ and $vu\notin E(D)$, then $h(v)\geq h(u)+1$.
\item If $uv,vu\in E(D)$, then $|h(u)-h(v)|\leq 1$.
\item If $uv,vu\in E(D)$ and $h(u)=h(v)=t$, then every path of order $t$ ending at $u$ contains $v$, and in fact contains $v\rightarrow u$ as its final arc. Similarly, every path of order $t$ ending at $v$ contains $u$, and in fact contains $u\rightarrow v$ as its final arc.
\end{enumerate}
\end{lemma}

\begin{proof}
{\bf (1).}\ Let $P$ be a longest path ending at $u$. If $v\notin V(P)$, then appending the arc $u\rightarrow v$ gives a path ending at $v$ of order $h(u)+1$, and hence $h(v)\geq h(u)+1$. If $v\in V(P)$, then the segment of $P$ from $v$ to $u$, together with the arc $u\rightarrow v$, forms a cycle. Since $D$ is semi-acyclic and $vu\notin E(D)$, this is impossible. Thus necessarily $v\notin V(P)$, and the assertion follows.

\smallskip
{\bf (2).}\ Suppose $uv,vu\in E(D)$, i.e., $u\leftrightarrow v$. Let $P$ be a longest path ending at $u$. If $v\notin V(P)$, then appending $u\rightarrow v$ gives $h(v)\geq h(u)+1$ and in particular \(h(v)\ge h(u)-1\).

 If $v\in V(P)$, then the segment of $P$ from $v$ to $u$, together with $u\rightarrow v$, is a cycle. Since $D$ is semi-acyclic, this cycle must be the $2$-cycle $v\rightarrow u\rightarrow v$. Hence $v$ occurs immediately before $u$ in $P$, and the initial segment of $P$ ending at $v$ has order $h(u)-1$. Thus $h(v)\geq h(u)-1$. 

Therefore, in all cases, $h(v)\ge h(u)-1$, or equivalently, $h(u)-h(v)\le 1$.

Since the hypothesis \(uv,vu\in E(D)\) is symmetric in \(u\) and \(v\),
the same argument with the roles of \(u\) and \(v\) interchanged gives $h(v)-h(u)\le 1$. 
Hence we have
\[
        |h(u)-h(v)|\le 1.
\]

\smallskip
{\bf (3).}\ Assume $h(u)=h(v)=t$. Let $P$ be a path of order $t$ ending at $u$. If $v\notin V(P)$, then appending $u\rightarrow v$ gives a path of order $t+1$ ending at $v$, contradicting $h(v)=t$. Therefore $v\in V(P)$. As in (2), the segment of $P$ from $v$ to $u$, together with $u\rightarrow v$, must be the $2$-cycle itself. Hence $v\rightarrow u$ is the final arc of $P$. The proof for paths of order $t$ ending at $v$ is symmetric.
\end{proof}

Now we wish to explore the case (2) of Lemma~\ref{lem:ha2} in more detail.

\begin{lemma}\label{lem:unequal}
Let $D$ be a semi-acyclic digraph, and let $u\leftrightarrow v$ be a $2$-cycle in $D$ such that $h(u)\neq h(v)$. Let $P$ be a longest path in $D$ ending at $u$ with $v\notin V(P)$, and let $Q$ be a longest path in $D$ ending at $v$ with $u\notin V(Q)$. Then the following hold:
\begin{enumerate}
\item $h(v)=h(u)+1$ if and only if $|P|>|Q|$.
\item Suppose $h(v)=h(u)+1$. Let $R$ be a longest path in $D$ starting at $u$ with $v\notin V(R)$, and let $S$ be a longest path in $D$ starting at $v$ with $u\notin V(S)$. Then the path obtained from $P$ followed by $R$ is at least as long as the corresponding path obtained from $Q$ followed by the arc $v\rightarrow u$ and then by $R$. Moreover, the path obtained from $P$ followed by the arc $u\rightarrow v$ and then by $S$ is longer than the corresponding path obtained from $Q$ followed by $S$.
\end{enumerate}
\end{lemma}

\begin{proof}
{\bf (1).}\ We first note that any path ending at $u$ contains $v$ if and only if it contains the arc $v\rightarrow u$ as its final arc; this follows from the proof of Lemma~\ref{lem:ha2}(2). Similarly, any path ending at $v$ contains $u$ if and only if it contains the arc $u\rightarrow v$ as its final arc. Therefore, by Definition~\ref{def:h}, we have
$$
h(u)=\max\{|P|,|Q|+1\}
\qquad\text{and}\qquad
h(v)=\max\{|Q|,|P|+1\}.
$$
If $h(v)=h(u)+1$, then the above equalities imply $h(u)=|P|$ and $h(v)=|P|+1$. Hence $|P|\geq |Q|+1$, and so $|P|>|Q|$. Conversely, if $|P|>|Q|$, then $|P|\geq |Q|+1$, whence
$$
h(u)=|P|\qquad\text{and}\qquad h(v)=|P|+1.
$$
Thus $h(v)=h(u)+1$.

\smallskip
{\bf (2).}\ We first observe that the concatenations used below are indeed paths. For example, suppose that $P$ and $R$ have a common vertex $w\neq u$. Then the segment of $P$ from $w$ to $u$, followed by the segment of $R$ from $u$ to $w$, forms a  cycle. Since $D$ is semi-acyclic, this cycle must be a $2$-cycle. But it contains $u$, which already belongs to the $2$-cycle $u\leftrightarrow v$, while both $P$ and $R$ avoid $v$; this contradicts the vertex-disjointness of the $2$-cycles. The same argument applies to the other concatenations below.

Since $h(v)=h(u)+1$, part (1) gives $|P|>|Q|$. The path obtained by concatenating $P$ and $R$ has order $|P|+|R|-1$, while the corresponding path obtained by concatenating $Q$, the arc $v\rightarrow u$, and $R$ has order $|Q|+|R|$. Since $|P|>|Q|$, we have
$$
|P|+|R|-1\geq |Q|+|R|.
$$
Similarly, the path obtained by concatenating $P$, the arc $u\rightarrow v$, and $S$ has order $|P|+|S|$, whereas the path obtained by concatenating $Q$ and $S$ has order $|Q|+|S|-1$. Since $|P|>|Q|$, we have
$$
|P|+|S|>|Q|+|S|-1.
$$
This proves the two asserted comparisons.
\end{proof}

In the following we prove the path partition conjecture holds for semi-acyclic digraphs.

\begin{theorem}\label{thm:semi-acyclic-ppc}
Let \(D=(V,E)\) be a semi-acyclic digraph. Then \(D\) satisfies the
path partition conjecture.
\end{theorem}

\begin{proof}
Let
\[
        m=\lambda(D),
\]
and let \(\lambda_1,\lambda_2\) be positive integers such that
\[
        \lambda_1+\lambda_2=m.
\]
Put \(q=\lambda_1\), so that \(\lambda_2=m-q\). Let \(h\) be the height
function defined in Definition~\ref{def:h}.

We first define a preliminary partition. Let
\[
        A_0=\{x\in V: h(x)\le q\}
\]
and
\[
        B_0=V\setminus A_0=\{x\in V: h(x)\ge q+1\}.
\]

We now modify this partition only at \(2\)-cycles near the boundary.

\textbf{Rule I.}

Let \(\mathcal C_1\) be the set of all \(2\)-cycles \(x\leftrightarrow y\)
such that
\[
        h(x)=h(y)=q+1.
\]
For each \(2\)-cycle \(C=\{x,y\}\in\mathcal C_1\), choose one endpoint (arbitrarily),
denoted \(c_A\), to be placed in \(A\), and let the other endpoint,
denoted \(c_B\), remain in \(B\).

\textbf{Rule II.}

Let \(\mathcal C_2\) be the set of all \(2\)-cycles
\(x\leftrightarrow y\) such that
\[
        \{h(x),h(y)\}=\{q+1,q+2\}.
\]
For each such \(2\)-cycle, we put the endpoint of height \(q+2\) into
\(A\), and the endpoint of height \(q+1\) into \(B\).

Thus applying Rule I and Rule II, we define
\[
        A=A_0
          \cup \{c_A: C\in\mathcal C_1\}
          \cup \{z: z\in C\in\mathcal C_2,\ h(z)=q+2\},
\]
and put
\[
        B=V\setminus A.
\]
We shall prove that
\[
        \lambda(D[A])\le q
        \qquad\text{and}\qquad
        \lambda(D[B])\le m-q.
\]

First we prove that \(\lambda(D[A])\le q\). Suppose, to the contrary,
that \(D[A]\) contains a path \(P\) of order \(q+1\). Let \(u\) be the
terminal vertex of \(P\). Since \(P\) is also a path of \(D\), we have
\[
        h(u)\ge q+1.
\]
On the other hand, by the construction of \(A\), every vertex of \(A\)
has height at most \(q+2\), and the only vertices of \(A\) having height
larger than \(q\) arise from the two boundary cases Rule I and Rule II above.

If \(h(u)=q+1\), then \(u\) was placed in \(A\) as one endpoint of a
\(2\)-cycle \(u\leftrightarrow v\) with
\[
        h(u)=h(v)=q+1.
\]
Then, by Lemma~\ref{lem:ha2}(3), every path of order \(q+1\) ending at \(u\)
contains \(v\), and in fact contains the arc \(v\to u\) as its final
arc. But \(v\in B\), contradicting \(P\subseteq D[A]\).

It remains to consider the case \(h(u)=q+2\). Then \(u\) was placed in
\(A\) because it lies in a \(2\)-cycle \(u\leftrightarrow v\) with
\[
        h(v)=q+1,\qquad h(u)=q+2.
\]
Then, by Lemma~\ref{lem:unequal}(1), applied with the roles of \(u\) and
\(v\) interchanged, every path ending at \(u\) and avoiding \(v\) has
order strictly smaller than \(q+1\).

Since \(v\in B\), the
path \(P\subseteq D[A]\) avoids \(v\), which contradicts \(|P|=q+1\).
Therefore
\[
        \lambda(D[A])\le q.
\]

It remains to prove that \(\lambda(D[B])\le m-q\).

\begin{claim}
Let \(R=(r_1,r_2,\ldots,r_s)\) be a path in \(D[B]\). Then \(D\)
contains a path of order at least \(q+s\).
\end{claim}

\begin{proof}[Proof of the claim]
Since \(r_1\in B\), we have \(h(r_1)\ge q+1\). Let \(P\) be a path of
order \(q+1\) ending at \(r_1\).

If \(P\) and \(R\) meet only at \(r_1\), then their concatenation is a
path of order
\[
        (q+1)+s-1=q+s,
\]
as required.

If the arc by which \(P\) enters \(r_1\) is not the reverse of the first
arc of \(R\), then no such obstruction occurs, and the preceding
concatenation works. In particular, this covers the case where \(P\)
enters \(r_1\) through an ordinary one-directional arc crossing from
\(A_0\) to \(B_0\).

Hence we need only consider the
case where the first arc of \(R\) is part of a \(2\)-cycle.

Suppose that \(P\) meets \(R\) at a vertex other than \(r_1\). Taking
the first such common vertex along \(R\), the corresponding segment of
\(P\) together with the initial segment of \(R\) forms a cycle. Since
\(D\) is semi-acyclic, this cycle must be a \(2\)-cycle. Thus the
obstruction can occur only when
\[
        r_1\leftrightarrow r_2
\]
and \(P\) uses the arc \(r_2\to r_1\) as its final arc.

Write
\[
        u=r_1,\qquad v=r_2.
\]
Thus \(R\) begins with \(u\to v\), while \(P\) ends with \(v\to u\).

First suppose that \(h(u)=h(v)=q+1\). Then \(u\leftrightarrow v\)
belongs to \(\mathcal C_1\), so Rule I splits \(u\) and \(v\)
between \(A\) and \(B\). This contradicts \(u,v\in B\).

Next suppose that
\[
        \{h(u),h(v)\}=\{q+1,q+2\}.
\]
Then \(u\leftrightarrow v\) belongs to \(\mathcal C_2\), so Rule II puts the endpoint of height \(q+2\) into \(A\). Again this
contradicts \(u,v\in B\).

It remains to consider the case
\[
        h(u)>q+1,\qquad h(v)>q+1.
\]
In this case the obstruction can be bypassed.

We shall use, in each of the following cases, the same observation as in
the proof of Lemma~\ref{lem:unequal}. The concatenations considered
below are indeed paths. For if the chosen initial path and the tail
\(r_3,r_4,\ldots,r_s\) had a common vertex, then the appropriate segment
of the initial path together with the corresponding segment of \(R\)
would form a cycle. Since \(D\) is semi-acyclic, this cycle would have to
be a \(2\)-cycle; but it would contain one of \(u\) or \(v\), already
belonging to the \(2\)-cycle \(u\leftrightarrow v\), a contradiction to
the vertex-disjointness of the \(2\)-cycles.

If \(h(u)=h(v)=t\), then by Lemma~\ref{lem:ha2}(3), there is a path
\(P_1\) of order \(t\) ending at \(v\), and \(P_1\) contains the arc
\(u\to v\) as its final arc. Since \(t>q+1\), we have \(t\ge q+2\).
Now concatenate \(P_1\) with the tail
\[
        r_3,r_4,\ldots,r_s
\]
of \(R\) using the arc $v\rightarrow r_3$ when $s\geq 3$. This gives a path of order
\[
        t+(s-2)\ge (q+2)+(s-2)=q+s.
\]

Now suppose \(h(v)>h(u)\). Since \(u\leftrightarrow v\), Lemma~\ref{lem:ha2}(2)
gives 
\(
        h(v)=h(u)+1.
\) 
By Lemma~\ref{lem:unequal}(1), there is a longest path \(P_1\) ending
at \(u\) and avoiding \(v\), with 
\(
        |P_1|=h(u)=h(v)-1.
\) 
We concatenate \(P_1\), the arc \(u\to v\), and the tail
\[
        r_3,r_4,\ldots,r_s.
\]
The resulting path has order
\[
        |P_1|+1+(s-2)=h(v)+s-2.
\]
Since \(h(u)>q+1\), we have \(h(v)>q+2\), and hence this order is at
least \(q+s\).

Finally suppose \(h(u)>h(v)\). Then 
\(
        h(u)=h(v)+1.
\) 
Again by Lemma~\ref{lem:unequal}(1), there is a longest path \(P_1\)
ending at \(v\) and avoiding \(u\), with 
\(
        |P_1|=h(v).
\) 
Concatenating \(P_1\) with the tail
\[
        r_3,r_4,\ldots,r_s
\]
using the arc $v\rightarrow r_3$ when $s\geq 3$ gives a path of order
\[
        h(v)+(s-2).
\]
Since \(h(v)>q+1\), this is at least
\[
        (q+2)+(s-2)=q+s.
\]

In each case, \(D\) contains a path of order at least \(q+s\). This
proves the claim.
\end{proof}

Now let \(R\) be a longest path in \(D[B]\). By the claim, \(D\) contains
a path of order at least
\[
        q+|R|.
\]
Since \(m=\lambda(D)\), we get 
\(
        q+|R|\le m.
\) 
Hence 
\(
        |R|\le m-q.
\) 
Therefore 
\(
        \lambda(D[B])\le m-q.
\) 

Combining the two inequalities, we have obtained a partition
\(V=A\cup B\) such that
\[
        \lambda(D[A])\le q=\lambda_1
\qquad 
and
\qquad 
        \lambda(D[B])\le m-q=\lambda_2.
\]
Thus \(D\) satisfies the path partition conjecture.
\end{proof}

\begin{example}\label{ex:rule-I-needed}
Let \(D\) be the digraph with vertex set
\[
        V(D)=\{1,2,\ldots,6\}
\]
and arc set
\[
E(D)=\{
1\to2,\ 1\to3,\ 2\to3,\ 3\to2,\ 2\to4,\ 3\to4,\ 
4\to5,\ 4\to6,\ 5\to6,\ 6\to5
\}.
\]
The only cycles of \(D\) are the two \(2\)-cycles
\[
        2\leftrightarrow 3,\qquad 5\leftrightarrow 6.
\]
which are vertex-disjoint. Thus \(D\) is semi-acyclic. The height values are
\[
\begin{array}{c|cccccc}
v&1&2&3&4&5&6\\ \hline
h(v)&1&3&3&4&6&6 .
\end{array}
\]
Hence \(\lambda(D)=6\).

The construction in the proof of Theorem~\ref{thm:semi-acyclic-ppc}
gives the following partitions.
\[
\begin{array}{c|c|c|c|c}
q & A & B & \lambda(D[A]) & \lambda(D[B])\\ \hline
1 &
\{1\} &
\{2,3,4,5,6\} &
1 & 5\\
2 &
\{1,2\} &
\{3,4,5,6\} &
2 & 4\\
3 &
\{1,2,3\} &
\{4,5,6\} &
3 & 3\\
4 &
\{1,2,3,4\} &
\{5,6\} &
4 & 2\\
5 &
\{1,2,3,4,5\} &
\{6\} &
5 & 1
\end{array}
\]
In each row,
\[
        \lambda(D[A])\le q,\qquad \lambda(D[B])\le 6-q.
\]

The case \(q=2\) illustrates Rule I. The preliminary height cut gives
\[
        A_0=\{1\},\qquad B_0=\{2,3,4,5,6\}.
\]
But \(D[B_0]\) contains the path
\[
        2\to3\to4\to5\to6,
\]
which has order \(5>6-2\). Thus the preliminary height cut fails.

Now the \(2\)-cycle
\[
        2\leftrightarrow3
\]
satisfies
\[
        h(2)=h(3)=3=q+1.
\]
By Rule I, the two endpoints are split between \(A\) and \(B\). Taking
\(2\in A\) and \(3\in B\), we get
\[
        A=\{1,2\},\qquad B=\{3,4,5,6\}.
\]
Then
\[
        \lambda(D[A])=2,\qquad \lambda(D[B])=4.
\]

We could also make the opposite choice in Rule I. Namely, instead of
putting \(2\) into \(A\), we may put \(3\) into \(A\) and leave \(2\) in
\(B\). Then
\[
        A=\{1,3\},\qquad B=\{2,4,5,6\}.
\]
In this case also,
\[
        \lambda(D[A])=2,\qquad \lambda(D[B])=4.
\]
Thus this example shows why Rule I is needed, and also shows that the
choice of $c_A$ in Rule I is arbitrary.
\end{example}

\begin{example}
Let \(D\) be the digraph with vertex set
\[
        V(D)=\{1,2,\ldots,8\}
\]
and arc set
{\small \[
E(D)=\{
1\to3,\ 1\to5,\ 2\to5,\ 2\to7,\ 3\to6,\ 3\to7,\ 3\to8,\ 
4\to7,\ 5\to7,\ 5\to8,\ 6\to3,\ 5\to1,\ 7\to4
\}.
\]}
The only cycles of \(D\) are the three \(2\)-cycles
\[
        1\leftrightarrow5,\qquad
        3\leftrightarrow6,\qquad
        4\leftrightarrow7,
\]
which are vertex-disjoint. Thus \(D\) is semi-acyclic. The height values are
\[
\begin{array}{c|cccccccc}
v&1&2&3&4&5&6&7&8\\ \hline
h(v)&3&1&4&6&2&5&5&5 .
\end{array}
\]
Hence \(\lambda(D)=6\).

The construction in the proof of Theorem~8.4 gives the following partitions.
\[
\begin{array}{c|c|c|c|c}
q & A & B & \lambda(D[A]) & \lambda(D[B])\\ \hline
1 &
\{1,2\} &
\{3,4,5,6,7,8\} &
1 & 4\\
2 &
\{2,5\} &
\{1,3,4,6,7,8\} &
2 & 4\\
3 &
\{1,2,5,6\} &
\{3,4,7,8\} &
3 & 3\\
4 &
\{1,2,3,4,5\} &
\{6,7,8\} &
4 & 1\\
5 &
\{1,2,3,5,6,7,8\} &
\{4\} &
5 & 1
\end{array}
\]
In each row,
\[
        \lambda(D[A])\le q,\qquad \lambda(D[B])\le 6-q.
\]

The case \(q=3\) is especially important. The preliminary height cut gives
\[
        A_0=\{1,2,5\},\qquad B_0=\{3,4,6,7,8\}.
\]
However, \(D[B_0]\) contains the path
\[
        6\to3\to7\to4,
\]
which has order \(4>6-3\). Thus the preliminary height cut fails.

Rule II detects the \(2\)-cycle
\[
        3\leftrightarrow6
\]
with
\[
        h(3)=4=q+1,\qquad h(6)=5=q+2.
\]
Therefore the endpoint of height \(q+2\), namely \(6\), is moved into
\(A\). This gives
\[
        A=\{1,2,5,6\},\qquad B=\{3,4,7,8\},
\]
and now
\[
        \lambda(D[A])=3,\qquad \lambda(D[B])=3.
\]
Thus this example shows why Rule II is necessary.
\end{example}

% Continuation of Section 8 after the definitions of semi-acyclic digraphs
% and the path partition conjecture have already been stated.

We now proceed for proving the path partition conjecture for the extension digraph with a semi-acyclic base.

\begin{definition}\label{def:extn}
Let $D = H[\mathscr{X},\mathscr{Y},\mathscr{Z}]$ be a digraph obtained from 
a semi-acyclic digraph $H$ by substituting bipartite digraphs $T_i$ with
bipartition $(X_i,Y_i)$, $i \in I$ for the 2-cycles $\langle x_i,y_i\rangle$ in $H$ and 
digraphs $Z_j, j \in J$ for the vertices $z_j$ not on 2-cycles.
We call each $Z_j$ an {\em ordinary block}, each $T_i$ a {\em bipartite block} and 
$X_i$ or $Y_i$ as a {\em port} of $T_i$.

For an ordinary block \(Z_j\), define
\[
        p_{z_j}^{in}
        =\max\{|P|: P\subseteq D-Z_j,
        \ \operatorname{ter}(P)\text{ has an arc into }Z_j\},
\]
where the maximum is taken to be \(0\) if no such path exists and $\operatorname{ter}(P)$ denotes the end or terminal vertex of a path $P$. Also define
\[
        p_{z_j}^{end}
        =\max\{|P|: P\text{ is a path in }D\text{ ending in }Z_j\}.
\]

For a bipartite block \(T_i=X_i\cup Y_i\), define
\[
        p_{X_i}^{in}
        =\max\{|P|: P\subseteq D-T_i,
        \ \operatorname{ter}(P)\text{ has an arc into }X_i\},
\]
where the maximum is taken to be \(0\) if no such path exists, and define
\(p_{Y_i}^{in}\) similarly. Also define
\[
        p_{X_i}^{end}
        =\max\{|P|: P\text{ is a path in }D\text{ ending in }X_i\},
\]
\[
        p_{Y_i}^{end}
        =\max\{|P|: P\text{ is a path in }D\text{ ending in }Y_i\}.
\]
Finally put
\[
        p_i^{in}=\max\{p_{X_i}^{in},p_{Y_i}^{in}\},
        \qquad
        p_i^{end}=\max\{p_{X_i}^{end},p_{Y_i}^{end}\}.
\]
\end{definition}

\begin{lemma}[Block-intersection lemma]\label{lem:block-intersection}
Let \(W\) be one of the substituted blocks of \(D\), that is, either
\(W=V(Z_j)\) for a vertex not on a 2-cycle, or \(W=V(T_i)\) for a
\(2\)-cycle block \(T_i\). If \(P\) is a path in \(D\), then the vertices
of \(P\) lying in \(W\) occur consecutively along \(P\). Consequently,
once \(P\) leaves \(W\), it cannot enter \(W\) again.
\end{lemma}

\begin{proof}
Suppose, to the contrary, that \(P\) leaves \(W\) and later enters \(W\)
again. Project the corresponding portion of \(P\) to the skeleton \(H\):
vertices of ordinary blocks \(Z_j\) are projected to \(z_j\), and vertices
of a bipartite block \(T_i=X_i\cup Y_i\) are projected to the
corresponding vertex \(x_i\) or \(y_i\). The projected walk starts in the
vertex, or in one of the two vertices, corresponding to \(W\), leaves it,
and later returns to it. Hence this directed walk contains a directed
cycle in \(H\).

Since \(H\) is semi-acyclic, this cycle must be a directed \(2\)-cycle.
However, the walk has left \(W\) and returned to it through vertices
outside \(W\). Thus the resulting cycle either has length greater than
\(2\), or is a \(2\)-cycle. In the latter case, if \(W=V(Z_j)\), then
\(z_j\) lies in a \(2\)-cycle of \(H\), contrary to the choice of \(z_j\). If
\(W=V(T_i)\), then this \(2\)-cycle shares a vertex with the original
\(2\)-cycle \(x_i\leftrightarrow y_i\), contradicting the
vertex-disjointness of the \(2\)-cycles.

Therefore \(P\) cannot leave \(W\) and later enter \(W\) again. Hence the
vertices of \(P\) lying in \(W\) occur consecutively along \(P\).
\end{proof}

\begin{lemma}\label{lem:pin-pend}
For every ordinary block \(Z_j\), 
\(
        0\le p_{z_j}^{in}<p_{z_j}^{end}.
\) 
For every bipartite block \(T_i=X_i\cup Y_i\), 
\(
        0\le p_{X_i}^{in}<p_{X_i}^{end},
        \ 
        0\le p_{Y_i}^{in}<p_{Y_i}^{end}.
\) 
Consequently,
\(
        0\le p_i^{in}<p_i^{end}.
\)
\end{lemma}

\begin{proof}
The non-negativity is clear from the definitions.

Let \(Z_j\) be an ordinary block. If \(p_{z_j}^{in}=0\), then
\(p_{z_j}^{end}\ge1\), since \(Z_j\ne\emptyset\). If \(p_{z_j}^{in}>0\), let
\(P\subseteq D-Z_j\) be a path of order \(p_{z_j}^{in}\) whose terminal vertex
has an arc into \(Z_j\). Since \(P\) is disjoint from \(Z_j\), we may
append one vertex of \(Z_j\) to \(P\). Thus 
\(
        p_{z_j}^{end}\ge p_{z_j}^{in}+1,
\) 
and hence \(p_{z_j}^{in}<p_{z_j}^{end}\).

The same argument applies to each port, namely, 
\(
        p_{X_i}^{in}<p_{X_i}^{end},
        \ 
        p_{Y_i}^{in}<p_{Y_i}^{end}.
\) 

Finally, if \(p_i^{in}=p_{X_i}^{in}\), then 
\[
        p_i^{end}
        =\max\{p_{X_i}^{end},p_{Y_i}^{end}\}
        \ge p_{X_i}^{end}
        >p_{X_i}^{in}
        =p_i^{in}.
\]
The case \(p_i^{in}=p_{Y_i}^{in}\) is symmetric. Hence
\(p_i^{in}<p_i^{end}\).
\end{proof}

\begin{definition}[The partition]\label{def:partition-substitution}
Let $D = H[\mathscr{X},\mathscr{Y},\mathscr{Z}]$ be a digraph as defined in Definition \ref{def:extn}. Fix an integer \(q\) with
\[
        1\le q<\lambda(D).
\]
For each ordinary block \(Z_j\), exactly one of the following three cases
holds:
\[
        p_{z_j}^{end}\le q,
        \qquad
        p_{z_j}^{in}\ge q,
        \qquad
        p_{z_j}^{in}<q<p_{z_j}^{end}.
\]
Define
\[
        S_0=\{{z_j}:p_{z_j}^{end}\le q\},
        \qquad
        R_0=\{{z_j}:p_{z_j}^{in}\ge q\},
        \qquad
        J_0=\{{z_j}:p_{z_j}^{in}<q<p_{z_j}^{end}\}.
\]
If \({z_j}\in S_0\), put \(V(Z_j)\subseteq A\). If \({z_j}\in R_0\), put
\(V(Z_j)\subseteq B\). If \({z_j}\in J_0\), set
\[
        q_{z_j}=q-p_{z_j}^{in}.
\]
Then \(1\le q_{z_j}<\lambda(Z_j)\). Indeed, \(p_{z_j}^{in}<q\) gives
\(q_{z_j}\ge1\). Also, by Lemma~\ref{lem:block-intersection}, every path
ending in \(Z_j\) has a final segment inside \(Z_j\) and an initial
segment outside \(Z_j\) of order at most \(p_{z_j}^{in}\). Hence
\[
        p_{z_j}^{end}\le p_{z_j}^{in}+\lambda(Z_j),
\]
and since \(q<p_{z_j}^{end}\), we get \(q_{z_j}<\lambda(Z_j)\). Since \(Z_j\)
satisfies the path partition conjecture, choose a partition
\[
        V(Z_j)=A_{z_j}\cup B_{z_j}
\]
such that
\[
        \lambda(Z_j[A_{z_j}])\le q_{z_j},
        \qquad
        \lambda(Z_j[B_{z_j}])\le \lambda(Z_j)-q_{z_j}.
\]
Put \(A_{z_j}\subseteq A\) and \(B_{z_j}\subseteq B\).\\[0.5em]

Now let \(T_i=X_i\cup Y_i\) be a bipartite block. If 
\(
        p_i^{end}\le q,
\) 
put \(V(T_i)\subseteq A\).\\[1em]
If
\(
        p_i^{in}<q<p_i^{end},
\) 
choose a port \(U_i\in\{X_i,Y_i\}\) such that 
\(
        p_{U_i}^{end}=p_i^{end}>q.
\) 
Let \(W_i\) be the other port, and put 
\(
        W_i\subseteq A,
        \ 
        U_i\subseteq B.
\)\\[1em]
Finally suppose 
\(
        p_i^{in}\ge q.
\)  

If 
\(
        p_{X_i}^{in}<q\le p_{Y_i}^{in},
\) 
put 
\(
        X_i\subseteq A,
        \ 
        Y_i\subseteq B.
\) 
If 
\(
        p_{Y_i}^{in}<q\le p_{X_i}^{in},
\) 
put 
\(
        Y_i\subseteq A,
        \ 
        X_i\subseteq B.
\)\\[1em]
If 
\(
        p_{X_i}^{in}\ge q
        \ \text{and}\ 
        p_{Y_i}^{in}\ge q,
\) 
put \(V(T_i)\subseteq B\).  
This completes the definition of \(A\) and \(B\).
\end{definition}

\begin{remark}\label{rem:genuine-concat}
In the proofs below, whenever we concatenate paths through a substituted
block \(W\), the outside pieces are vertex-disjoint. Indeed, if an
outside vertex occurred in two pieces, then, by taking the first such
repetition, the projected walk in \(H\) would leave \(W\) and later return
to it. Lemma~\ref{lem:block-intersection} then gives a forbidden directed
cycle in \(H\). Thus the displayed concatenations are genuine paths.
\end{remark}

Finally we prove Theorem \ref{main3} which is reiterated in the following form:

\begin{theorem}\label{thm:semi-acyclic-substitution-ppc}
Let $D = H[\mathscr{X},\mathscr{Y},\mathscr{Z}]$ be a digraph defined as in Definition \ref{def:extn}. 
Suppose that every ordinary block \(Z_j\) satisfies the path partition conjecture. 
Then \(D\) satisfies the path partition conjecture.
\end{theorem}
\begin{proof}
Fix an integer \(q\) with 
\(
        1\le q<\lambda(D).
\)
Define the partition \(V(D)=A\cup B\) as in Definition \ref{def:partition-substitution}.\\[0.5em]

We first show that 
\[
        \lambda(D[A])\le q.
\]

Let \(P\) be a path in \(D[A]\), and let \(u\) be its terminal vertex. We
consider the substituted block containing \(u\). 
First suppose \(u\in V(Z_j)\) for some ordinary block \(Z_j\). If
\({z_j}\in S_0\), then by definition
\[
        |P|\le p_{z_j}^{end}\le q.
\]
If \({z_j}\in J_0\), then Lemma~\ref{lem:block-intersection} implies that the
vertices of \(P\) lying in \(Z_j\) form a final segment of \(P\). Thus 
\(
        P=P'P_{z_j},
\) 
where \(P_{z_j}\subseteq Z_j[A_{z_j}]\). Hence
\[
        |P'|\le p_{z_j}^{in},
        \qquad
        |P_{z_j}|\le \lambda(Z_j[A_{z_j}])\le q_{z_j}.
\]
Therefore
\[
        |P|\le p_{z_j}^{in}+q_{z_j}=q.
\]
The case \({z_j}\in R_0\) cannot occur, because \(V(Z_j)\subseteq B\).\\[0.5em]

Now suppose \(u\in V(T_i)\), where \(T_i=X_i\cup Y_i\). If
\(p_i^{end}\le q\), then \(V(T_i)\subseteq A\), and hence
\[
        |P|\le p_i^{end}\le q.
\]
Suppose next that 
\(
        p_i^{in}<q<p_i^{end}.
\) 
By construction, 
\(
        W_i\subseteq A,
        \ 
        U_i\subseteq B.
\) 

Thus \(u\in W_i\). Since \(T_i\) is bipartite, there are no arcs inside
the port \(W_i\). Hence the final segment of \(P\) inside \(T_i\) consists
only of the vertex \(u\). Therefore
\[
        P=P'u,
\]
where \(P'\subseteq D-T_i\), and the terminal vertex of \(P'\), if
\(P'\ne\emptyset\), has an arc into \(W_i\). Hence
\[
        |P'|\le p_{W_i}^{in}\le p_i^{in}<q.
\]
Since \(|P'|\) is an integer, \(|P'|\le q-1\), and so
\[
        |P|=|P'|+1\le q.
\]

Finally suppose \(p_i^{in}\ge q\). If the construction puts one port in
\(A\), say \(X_i\subseteq A\) and \(Y_i\subseteq B\), then necessarily
\(p_{X_i}^{in}<q\). The same argument as above shows that any path in
\(D[A]\) ending in \(X_i\) has order at most \(q\). The case
\(Y_i\subseteq A\) is symmetric. If both port-in values are at least
\(q\), then \(V(T_i)\subseteq B\), so \(u\notin V(T_i)\).

Therefore every path in \(D[A]\) has order at most \(q\), and the claim
follows.\\[0.5em]

Next we prove that 
\[
        \lambda(D[B])\le \lambda(D)-q.
\]

Let
\[
        R=(r_1,r_2,\ldots,r_s)
\]
be a longest path in \(D[B]\). It is enough to show that \(D\) contains a
path of order at least \(q+s\). Let \(W\) be the substituted block
containing \(r_1\).

First suppose \(W=V(Z_j)\) for some ordinary block \(Z_j\). The case
\({z_j}\in S_0\) cannot occur, since \(V(Z_j)\subseteq A\). If \({z_j}\in R_0\),
then \(p_{z_j}^{in}\ge q\). Choose a path \(P\subseteq D-Z_j\) of order
\(p_{z_j}^{in}\) whose terminal vertex has an arc into \(Z_j\). Since arcs
into \(Z_j\) are complete by the definition of the substitution, \(P\)
can be followed by \(R\). By Remark~\ref{rem:genuine-concat}, this
concatenation is a path. Hence
\[
        \lambda(D)\ge |P|+|R|\ge q+s.
\]

Now suppose \({z_j}\in J_0\). By Lemma~\ref{lem:block-intersection}, the
vertices of \(R\) lying in \(Z_j\) form an initial segment of \(R\). Write
\[
        R=R_{z_j}R',
\]
where \(R_{z_j}\subseteq Z_j[B_{z_j}]\), and \(R'\) is the remaining tail,
possibly empty. Then
\[
        |R_{z_j}|\le \lambda(Z_j[B_{z_j}])\le \lambda(Z_j)-q_{z_j}.
\]
Choose a path \(P\subseteq D-Z_j\) of order \(p_{z_j}^{in}\) whose terminal
vertex has an arc into \(Z_j\), and let \(L_{z_j}\) be a longest path in
\(Z_j\). By the substitution rule, \(P\), \(L_{z_j}\), and \(R'\) can be
concatenated to form a path in \(D\); if \(R'\) is empty, the last piece
is omitted. Thus
\[
        \lambda(D)\ge p_{z_j}^{in}+\lambda(Z_j)+|R'|.
\]
Consequently,
\[
\begin{aligned}
q+s
&=q+|R_{z_j}|+|R'|\\
&\le q+\lambda(Z_j)-q_{z_j}+|R'|\\
&=q+\lambda(Z_j)-(q-p_{z_j}^{in})+|R'|\\
&=p_{z_j}^{in}+\lambda(Z_j)+|R'|\\
&\le \lambda(D).
\end{aligned}
\]

Now suppose \(W=V(T_i)\), where \(T_i=X_i\cup Y_i\). If
\(p_i^{end}\le q\), then \(V(T_i)\subseteq A\), so this case cannot
occur. 
Suppose first that 
\[
        p_i^{in}<q<p_i^{end}.
\]
By construction,
\[
        W_i\subseteq A,
        \qquad
        U_i\subseteq B,
        \qquad
        p_{U_i}^{end}=p_i^{end}>q.
\]
Since \(R\subseteq D[B]\) starts in \(T_i\), we have \(r_1\in U_i\).
Moreover, \(U_i\) is an independent set, because it is one part of the
bipartition of \(T_i\). Hence the initial segment of \(R\) inside \(T_i\)
consists only of \(r_1\).

Let \(Q\) be a path in \(D\) ending in \(U_i\) of order
\[
        p_{U_i}^{end}=p_i^{end}\ge q+1.
\]
If \(s=1\), then \(|Q|\ge q+1=q+s\), and we are done. If \(s\ge2\), then
\(r_2\notin V(T_i)\). Since \(r_1\to r_2\), the substitution rule implies
that every vertex of \(U_i\) has an arc to \(r_2\). Thus the terminal
vertex of \(Q\) has an arc to \(r_2\), and by
Remark~\ref{rem:genuine-concat}, 
\(
        Qr_2r_3\cdots r_s
\) 
is a path in \(D\). Hence
\[
        \lambda(D)
        \ge |Q|+(s-1)
        \ge (q+1)+(s-1)=q+s.
\]

Finally suppose
\[
        p_i^{in}\ge q.
\]
If
\[
        p_{X_i}^{in}<q\le p_{Y_i}^{in},
\]
then by construction \(X_i\subseteq A\) and \(Y_i\subseteq B\). Thus
\(r_1\in Y_i\). Choose a path \(P\subseteq D-T_i\) of order at least \(q\)
whose terminal vertex has an arc into \(Y_i\). By the substitution rule,
the terminal vertex of \(P\) has an arc to \(r_1\), so \(P\) can be
followed by \(R\). Hence \(D\) contains a path of order at least
\(q+s\). The case
\[
        p_{Y_i}^{in}<q\le p_{X_i}^{in}
\]
is symmetric.

It remains to consider the case
\[
        p_{X_i}^{in}\ge q
        \qquad\text{and}\qquad
        p_{Y_i}^{in}\ge q.
\]
Then \(V(T_i)\subseteq B\). If \(r_1\in X_i\), choose a path
\(P\subseteq D-T_i\) of order at least \(q\) whose terminal vertex has an
arc into \(X_i\). If \(r_1\in Y_i\), choose such a path entering \(Y_i\).
In either case, \(P\) can be followed by \(R\), giving a path in \(D\) of
order at least \(q+s\).

In every case, \(D\) contains a path of order at least \(q+s\). Hence
\(q+s\le\lambda(D)\). Since \(s=|R|=\lambda(D[B])\), we get
\[
        \lambda(D[B])\le \lambda(D)-q.
\]

Thus, we have the partition
\(V(D)=A\cup B\) satisfying
\[
        \lambda(D[A])\le q,
        \qquad
        \lambda(D[B])\le \lambda(D)-q.
\]
Since \(q\) was arbitrary, \(D\) satisfies the path partition conjecture.
\end{proof}

\begin{remark}[Recovery of Rules I and II]\label{rem:rules-recovered}
Let us check that the port assignment for a bipartite block really
contains the earlier Rules I and II as special cases. Suppose that
\(T_i=X_i\cup Y_i\) is a singleton \(2\)-cycle, say
\[
        X_i=\{x\},\qquad Y_i=\{y\},
\]
with both arcs \(x\to y\) and \(y\to x\). Then
\[
        p_{X_i}^{end}=h(x),\qquad p_{Y_i}^{end}=h(y).
\]
Moreover, \(p_{X_i}^{in}+1\) is the maximum order of a path ending at
\(x\) and avoiding \(y\), and \(p_{Y_i}^{in}+1\) is the maximum order of
a path ending at \(y\) and avoiding \(x\).

First suppose
\[
        h(x)=h(y)=q+1.
\]
By Lemma~\ref{lem:ha2}(3), every path of order \(q+1\) ending at
\(x\) contains \(y\), in fact with final arc \(y\to x\). Hence every
path ending at \(x\) and avoiding \(y\) has order at most \(q\). On the
other hand, taking a path of order \(q+1\) ending at \(y\) and deleting
its final vertex \(y\), we obtain a path of order \(q\) ending at \(x\)
and avoiding \(y\). Therefore 
\(
        p_{X_i}^{in}=q-1.
\) 
Similarly, 
\(
        p_{Y_i}^{in}=q-1.
\) 
Thus 
\(
        p_i^{in}=q-1<q<p_i^{end}=q+1.
\)
So this situation lies in the middle case of the port assignment. Since
both ports have end-value \(q+1>q\), either port may be chosen as the
\(B\)-port. Consequently the construction separates \(X_i\) and \(Y_i\)
arbitrarily between \(A\) and \(B\), exactly as in Rule I.

Next suppose
\[
        h(x)=q+1,\qquad h(y)=q+2.
\]
Then 
\(
        p_{X_i}^{end}=q+1,\qquad p_{Y_i}^{end}=q+2.
\) 
By Lemma~\ref{lem:unequal}, the longest path ending at \(x\) and
avoiding \(y\) has order \(q+1\). Hence 
\(
        p_{X_i}^{in}=q.
\) 
On the other hand, every path ending at \(y\) and avoiding \(x\) has
order at most \(q\), and therefore 
\(
        p_{Y_i}^{in}\le q-1.
\) 
Thus 
\(
        p_i^{in}=q\ge q,
\) 
and this case lies in the final port-in part of the construction. Since 
\(
        p_{Y_i}^{in}<q\le p_{X_i}^{in},
\) 
the construction puts 
\(
        Y_i\subseteq A,\qquad X_i\subseteq B.
\) 
That is, the endpoint with larger height \(q+2\) is placed in \(A\), and
the endpoint with smaller height \(q+1\) is placed in \(B\), exactly as
in Rule II. The case 
\(
        h(y)=q+1,\ h(x)=q+2
\)
is symmetric.

Thus the earlier Rules I and II are not extra assumptions in the
substitution theorem. They appear naturally as the singleton \(2\)-cycle
cases of the present port assignment. Notice also that there is no
contradiction with Lemma~\ref{lem:pin-pend}: in the unequal case above,
for instance,
\[
        p_{X_i}^{in}=q<q+1=p_{X_i}^{end},
        \qquad
        p_{Y_i}^{in}\le q-1<q+2=p_{Y_i}^{end}.
\]
\end{remark}

\begin{corollary} \cite{BGLL}
The path partition conjecture holds for 3-quasi-transitive digraph.
\end{corollary}
\begin{proof}
Let $D$ be a 3-quasi-transitive digraph. Then by Theorem \ref{3-quasi-trans} 
$D = H[\mathscr{X},\mathscr{Y},\mathscr{Z}]$, that is, $D$ is obtained from a 
semi-acyclic 3-transitive digraph by substituting strong semi-complete bipartite 
digraphs, strong semicomplete digraphs or digraphs $F_n$. As the substituted digraphs 
satisfy the path partition conjecture, $D$ satisfies the conjecture by 
Theorem \ref{thm:semi-acyclic-substitution-ppc}.
\end{proof}

\begin{corollary}\label{cor:acyclic-composition-ppc} \cite{AG}
Let $D = H[S_1,S_2, \dots, S_r]$. Suppose that $H$ is acyclic and each $S_i$ satisfies 
the path partition conjecture. Then $D$ satisfies the path partition conjecture.
\end{corollary}
\begin{proof}
Since $H$ is acyclic, it is semi-acyclic and contains no directed
\(2\)-cycle. Therefore, in the substitution of
Theorem~\ref{thm:semi-acyclic-substitution-ppc}, there are no bipartite
blocks. Every substituted block is an ordinary block.

Thus the hypotheses of Theorem~\ref{thm:semi-acyclic-substitution-ppc}
are satisfied, because each ordinary block \(S_i\) satisfies the
path partition conjecture. Hence $D$ satisfies the path partition conjecture.
\end{proof}

\section{Conclusion}

In this paper, we characterize $3$-transitive digraphs and $3$-quasi-transitive digraphs as certain extensions over semi-acyclic base graphs. As an application, we proved the path partition conjecture for an extension of semi-acyclic digraphs where strong components obey the conjecture. This generalizes some known results on the conjecture. Following these constructions and methodology one may try to verify the conjecture, in general, for digraphs whose strong components satisfy path partition conjecture.

{\bf Acknowlegement}: We would like to thank Cesar~Hern$\acute{\text{a}}$ndez-Cruz for 
his interest and suggestions which led to the study of the path partition conjecture.

\bibliographystyle{amsplain}

\begin{thebibliography}{1}

\renewcommand{\baselinestretch}{1}

\bibitem{AG} A.~Arroyo and H. Galeana-S\'anchez,
\emph{The path partition conjecture is true for some generalizations of tournaments},
Discrete Math., \textbf{313} (2013) 293--300.

\bibitem{JBJ}
J.~Bang-Jensen, \emph{The structure of strong arc-locally semicomplete digraphs}, Discrete Mathematics, \textbf{283} (2004) 1--6.

\bibitem{BH} 
J.~Bang-Jensen and J.~Huang, \emph{Quasi-transitive digraphs}, 
J. Graph Theory \textbf{20} (1995) 141–-161.

\bibitem{BGLL} 
J.~Bo, Z.~Gao, X.~Lian, and J.~Liu, 
\emph{Some digraph classes that meet the directed path partition conjecture},
Discrete Applied Math., \textbf{372} (2025) 260--267.

\bibitem{CM}
P.~Camion, \emph{Chemins et circuits hamiltoniens des graphes complets}, C. R. Acad. Sci. Paris, \textbf{249} (1959), 2151--2152.

\bibitem{DGGS}
S.~Das, P.~Ghosh, S.~Ghosh and S.~Sen, \emph{Oriented bipartite graphs and the Goldbach graph}, Discrete Mathematics, \textbf{344} (2021), 112497.

%\bibitem{Frick} M.~Frick, \emph{A survey on the path partition conjecture},
%Discussiones Math., \textbf{33} (2-13) 117--131.

\bibitem{gallai} T. Gallai, \emph{Transitiv orientierbare graphen},
{Acta Math. Acad. Sci. Hungar} \textbf{18} (1967) 25--66.

\bibitem{ghouila} A. Ghouil\` a-Houri, \emph{Caract\' erisation des 
graphes non orient\` es dont on peut orienter les arr\^ etes 
de maniere \` a obtenir le graphe d'un relation d'ordre}, 
{C. R. Acad. Sci. Paris} \textbf{254} (1962) 1370--1371. 

\bibitem{HH-lex} P. Hell and J. Huang, \emph{Lexicographic orientation and representation
algorithms for comparability graphs, proper circular arc graphs, and proper interval 
graphs}, {J. Graph Theory}, \textbf{20} (1995) 361--374.
 
\bibitem{GGU}
H.~Galeana-S$\acute{\text{a}}$nchez, I.A.~Goldfeder and I.~Urrutia, \emph{On the structure of strong 3-quasi-transitive digraphs}, Discrete Mathematics, \textbf{310} (2010), 2495--2498.

\bibitem{HC}
C.~Hern$\acute{\text{a}}$ndez-Cruz, \emph{$3$-transitive digraphs}, Discussiones Mathematicae Graph Theory \textbf{32} (2012) 205--219.

\bibitem{H} J.~Huang, \emph{Lexicographic orientation algorithms for orientation 
completion problems}, Chapter in Classes of Digraphs (J.~Bang-Jensen and G.~Gutin eds),
Springer (2018).


\end{thebibliography}

\end{document}